\documentclass[11pt]{amsart}

\usepackage{amsmath, amssymb, amsfonts,enumerate}
\usepackage[all]{xy}
\usepackage{amscd}
\usepackage{comment}
\usepackage{mathtools}
\usepackage{tikz} 
\usepackage{tikz-cd} 
\tikzcdset{diagrams={nodes={inner sep=7.5pt}}}

\usepackage{url}
\usepackage{hyperref}

\newtheorem{theorem}{Theorem}[section]
\newtheorem{lemma}[theorem]{Lemma}
\newtheorem{corollary}[theorem]{Corollary}
\newtheorem{proposition}[theorem]{Proposition}
\newtheorem{theoremletter}{Theorem}

\newtheorem{conjecture}[theorem]{Conjecture}

 \theoremstyle{definition}
 \newtheorem{definition}[theorem]{Definition}

 \newtheorem{example}[theorem]{Example}

\numberwithin{equation}{section}
 
\newcommand {\N}{\mathbb{N}} 
\newcommand {\Z}{\mathbb{Z}}

\newcommand {\C}{\mathbb{C}}

\newcommand{\FF}{\mathcal{F}}
\newcommand{\GG}{\mathcal{G}}

\newcommand\sbullet[1][.5]{\mathbin{\vcenter{\hbox{\scalebox{#1}{$\bullet$}}}}}

\DeclareMathOperator{\mdim}{mdim}

\DeclareMathOperator{\Ker}{Ker}

\DeclareMathOperator{\im}{Im}

\DeclareMathOperator{\Id}{Id}

\begin{document}

\title[Symbolic group varieties and dual surjunctivity]{On symbolic group varieties and dual surjunctivity}
\author[Xuan Kien Phung]{Xuan Kien Phung}
\address{Département de mathématiques, Université du Québec à Montréal, Case postale 8888, Succursale Centre-ville, Montréal, QC, H3C 3P8, Canada}
\email{phungxuankien1@gmail.com}
\begin{abstract}
Let $G$ be a group. Let $X$ be an algebraic group over an algebraically closed field $K$. Denote by $A=X(K)$ the set of rational points of $X$. We study algebraic group cellular automata $\tau \colon A^G \to A^G$ whose local defining map is induced by a homomorphism of algebraic groups $X^M \to X$ where $M$ is a finite memory. When $G$ is sofic and $K$ is uncountable, we show that if $\tau$ is post-surjective then it is weakly pre-injective. Our result extends the dual version of Gottschalk's Conjecture for finite alphabets proposed by Capobianco, Kari, and Taati. When $G$ is amenable, we prove that if $\tau$ is surjective then it is weakly pre-injective, and conversely, if $\tau$ is pre-injective then it is surjective. Hence, we obtain a complete answer to a question of Gromov on the Garden of Eden theorem in the case of algebraic group cellular automata.
\end{abstract} 
\date{\today}
\maketitle

\setcounter{tocdepth}{1}

\section{Introduction} 
We recall basic notations in symbolic dynamics. 
Fix a set $A$ called the \emph{alphabet},  and a group  $G$, the \emph{universe}.
A \emph{configuration} $c \in A^G$ is a map $c \colon G \to A$. 
The Bernoulli shift action $G \times A^G \to A^G$ is defined by $(g,c) \mapsto g c$, 
where $(gc)(h) \coloneqq  c(g^{-1}h)$ for  $g,h \in G$ and $c \in A^G$. 
For $\Omega \subset G$ and $c \in A^G$, the \emph{restriction} 
$c\vert_\Omega \in A^\Omega$ is given by $c\vert_\Omega(g) \coloneqq  c(g)$ for all $g \in \Omega$. 
\par
Following von Neumann \cite{neumann}, a \emph{cellular automaton} over  the group $G$ and the alphabet $A$ is a map
$\tau \colon A^G \to A^G$ admitting a finite \emph{memory set} $M \subset G$
and a \emph{local defining map} $\mu \colon A^M \to A$ such that 
\begin{equation*} 
\label{e:local-property}
(\tau(c))(g) = \mu((g^{-1} c )\vert_M)  \quad  \text{for all } c \in A^G \text{ and } g \in G.
\end{equation*} 
 \par
Two configurations $c,d  \in A^G$ are \emph{asymptotic} if $c\vert_{G \setminus E}=d\vert_{G \setminus E}$ for some finite subset $E \subset G$. 
Let $\tau \colon A^G \to A^G$ be a cellular automaton. Then $\tau$ is  \emph{pre-injective} if $\tau(c) = \tau(d)$ implies $c= d$ whenever $c, d \in A^G$ are asymptotic. We say that $\tau$ is \emph{post-surjective} if for every $x, y \in A^G$ with $y$ asymptotic to $\tau (x)$, we can find $z \in A^G$ asymptotic to $x$ such that  $ \tau(z)=y$.
\par 
The cellular automaton $\tau \colon A^G \to A^G$ is said to be \emph{linear} if $A$ is a finite-dimensional vector space and $\tau$ is a linear map. 
 \par 
The important Gottschalk's conjecture \cite{gottschalk} asserts that over any universe, an injective cellular automaton with finite alphabet must be surjective. 
\par 
The conjecture was shown to hold over sofic groups (cf. \cite{gromov-esav}, \cite{weiss-sgds}, see also \cite{csc-sofic-linear}, \cite{cscp-alg-ca},  \cite{phung-2020}) while no examples of non-sofic groups are known in the litterature. 
The dual version of Gottschalk's conjecture was introduced recently by 
Capobianco, Kari, and Taati in \cite{kari-post-surjective} and states the following: 

\begin{conjecture}
\label{c:dual-gottschalk-conjecture}
Let $G$ be a group and let $A$ be a finite set. Suppose that $\tau \colon A^G \to A^G$ is a post-surjective cellular automaton. Then $\tau$ is pre-injective.
\end{conjecture} 
\par 
As for Gottschalk's conjecture, the above dual surjunctivity conjecture is also known when the universe is a sofic group (cf. \cite[Theorem~2]{kari-post-surjective}): 
\begin{theorem}[Capobianco-Kari-Taati]
\label{t:kari-dual-surj}
Let $G$ be a sofic group and let $A$ be a finite set. Suppose that $\tau \colon A^G \to A^G$ is a post-surjective cellular automaton. Then $\tau$ is pre-injective.
\end{theorem} 
\par 
Moreover, as Bartholdi pointed out in \cite[Theorem~1.6]{bartholdi}, Conjecture~\ref{c:dual-gottschalk-conjecture} also holds for linear cellular automata over sofic groups: 
\begin{theorem}
\label{t:bartholdi-post-surjective}
Let $G$ be a sofic group and let $V$ be a finite dimensional vector space over a field. 
Suppose that $\tau \colon V^G \to V^G$ is a post-surjective linear cellular automaton. Then $\tau$ is pre-injective. 
\end{theorem} 
Several related applications of groups satisfying Conjecture~\ref{c:dual-gottschalk-conjecture} are investigated in the paper \cite{doucha}. 
\par 
Now let us fix a group $G$ and an algebraic group $X$ over an algebraically closed field $K$. Denote by $A \coloneqq X(K)$ the set of $K$-points of $X$. We regard $A\subset X$ as a subset which consists of closed points of $X$.
\par
The set of \emph{algebraic group cellular automata} over $(G,X,K)$, denoted by  $CA_{algr}(G,X,K)$, consists of  cellular automata $\tau \colon A^G \to A^G$ which  admit a memory set $M$ 
with local defining map $\mu \colon A^M \to A$ induced by some homomorphism of algebraic groups
$f \colon X^M \to X$, i.e., $\mu=f\vert_{A^M}$, 
where $X^M$ is the fibered product of copies of $X$ indexed by $M$. 
\par 
In \cite[Definition~8.1]{phung-2020}, two notions of weak pre-injectivity, namely,  $(\sbullet[.9])$-pre-injectivity and  $(\sbullet[.9]\sbullet[.9])$-pre-injectivity, are introduced  for the class $CA_{algr}$ (cf. Section~\ref{s:weak-pre-injectivity}). We prove in Corollary~\ref{c:moore-one-star-implies-two-star-pre-injectivity} that in  $CA_{algr}$, we have: 
\begin{equation} 
\label{e:intro-one-star-implies-two-star}
( \sbullet[.9])\mbox{-pre-injectivity} \implies (\sbullet[.9] \sbullet[.9])\mbox{-pre-injectivity}.
\end{equation} 
\par 
Note that for linear cellular automata,  pre-injectivity, $(\sbullet[.9])$-pre-injectivity, and  $(\sbullet[.9]\sbullet[.9])$-pre-injectivity are equivalent (cf.~\cite[Proposition~8.8]{phung-2020}). 
\par 
 Generalizing Theorem~\ref{t:bartholdi-post-surjective}, we establish Conjecture~\ref{c:dual-gottschalk-conjecture} for the class $CA_{algr}$  where the universe is a sofic group and the alphabet is an arbitrary algebraic group not necessarily connected  (cf.~Theorem~\ref{t:dual-gottschalk-ca-algr-full}). 
\begin{theoremletter}
\label{t:post-surj-full-ca-algr}
Let $G$ be a sofic group and let $X$ be an algebraic group over an uncountable algebraically closed field $K$. 
Suppose that $\tau \in CA_{algr}(G,X,K)$ is post-surjective. Then $\tau$ is   $(\sbullet[.9])$-pre-injective. 
\end{theoremletter} 
\par 
We observe in Example \ref{ex:post-sur-not-pre-inj} that for every group $G$ there exist a complex algebraic group $X$ and $\tau \in CA_{algr}(G,X,\C)$ such that $\tau$ is post-surjective but not pre-injective. Moreover, in  characteristic zero, we prove (Theorem~\ref{t:reversible}) that every post-surjective, pre-injective $\tau \in CA_{algr}$ is reversible with $\tau^{-1} \in CA_{algr}$. 
Such property was known in the litterature for cellular automata with finite alphabet \cite[Theorem~1]{kari-post-surjective} and for linear cellular automata \cite{bartholdi}. 
\par 
The classical Myhill-Moore Garden of Eden theorem for finite alphabets  (cf.~\cite{myhill}, \cite{moore}) asserts that 
a cellular automaton over the group universe $\Z^d$ is pre-injective if and only if it is surjective. Over amenable groups, the theorem was extended to cellular automata with finite alphabet in~\cite{ceccherini} and to linear cellular automata in~\cite{csc-garden-linear}. The theorem fails over non-amenable groups (cf. \cite{bartholdi}, \cite{bartholdi-kielak}, see also \cite{goe-survey}). 
At the end of \cite{gromov-esav}, Gromov asked
\begin{quote}
8.J. Question. Does the Garden of Eden theorem generalize to the 
proalgebraic category?
First, one asks if pre-injective  $\implies$  surjective, while the reverse implication 
needs further modification of definitions.
\end{quote}
\par 
Let  $G$ be an amenable group and let $K$ be an algebraically closed field. 
The papers \cite{cscp-alg-goe} and \cite{phung-2020} respectively give a positive answer to Gromov's question for the class $CA_{alg}(G,X,K)$ (cf.~Section \ref{s:alg-ca}) when $X$ is a complete irreducible algebraic variety over $K$ and for the class $CA_{algr}(G,X,K)$ when $X$ is a connected algebraic group over $K$. 
\par 
In this paper, we obtain the following complete answer to Gromov's question for the class $CA_{algr}(G,X,K)$ where $X$ is an arbitrary algebraic group  (cf.~Theorem~\ref{t:myhill-ca-algr-full}, Theorem~\ref{t:moore-ca-algr-full}). 

\begin{theoremletter}
 \label{t:gromov-answer-algr-full} 
 Let $G$ be an amenable group and let $X$ be an algebraic group over  an algebraically closed field $K$. Suppose that $\tau \in CA_{algr}(G,X,K)$.  Then the following hold: 
 \begin{enumerate} [\rm (i)]
     \item 
     If $\tau$ is pre-injective, then it is surjective; 
     \item 
     If $\tau$ is surjective, then it is both $(\sbullet[.9])$- pre-injective and $(\sbullet[.9] \sbullet[.9])$-pre-injective. 
 \end{enumerate}
\end{theoremletter}
\par 
    In Proposition~\ref{p:myhill-ca-algr-two-star-not-hold}, we show that one cannot replace the pre-injectivity hypothesis in Theorem~\ref{t:gromov-answer-algr-full}.(i) by the weaker $(\sbullet[.9] \sbullet[.9])$-pre-injectivity. Moreover, we obtain a very general result (cf.~Theorem~\ref{t:post-surjective-implies-surjective}) saying that post-surjectivity implies surjectivity in $CA_{algr}$ and $CA_{alg}$. Consequently, when the universe $G$ is an amenable group, Theorem~\ref{t:gromov-answer-algr-full}.(ii) implies Theorem~\ref{t:post-surj-full-ca-algr}. 
 \par 
The paper is organized as follows. In Section \ref{s:preliminary}, we present briefly important properties of sofic groups as well as amenable groups. Section~\ref{s:alg-ca} recalls basic definitions about the classes $CA_{alg}$ and $CA_{algr}$.  
In Section~\ref{s:induced-map-connected-component}, we introduce the useful tool of induced maps on the set of connected components of algebraic varieties and give some applications to the class $CA_{algr}$. 
Then in Section~\ref{s:weak-pre-injectivity}, we investigate at length $(\sbullet[.9])$-pre-injectivity and $(\sbullet[.9] \sbullet[.9])$-pre-injectivity in the class $CA_{algr}$ and prove \eqref{e:intro-one-star-implies-two-star}.   
In Section~\ref{s:uniform-post-surj}, we establish a certain uniform post-surjectivity property (Lemma~\ref{l:uniform-psot-surjectivity}) and show in particular that post-surjectivity implies surjectivity in  $CA_{alg}$ and $CA_{algr}$ (Theorem~\ref{t:post-surjective-implies-surjective}).  
We present the proof of Theorem~\ref{t:post-surj-full-ca-algr} in  Section~\ref{s:post-surjectivity}. Then 
Theorem~\ref{t:myhill-ca-algr-full} establishes the Myhill property for $CA_{algr}$ as stated in  Theorem~\ref{t:gromov-answer-algr-full}.(i). Finally, the Moore property for $CA_{algr}$, i.e.,  
Theorem~\ref{t:gromov-answer-algr-full}.(ii), is proved in Section~\ref{s:moore} (Theorem~\ref{t:moore-ca-algr-full}).

\section{Preliminaries}
\label{s:preliminary}
To simplify the presentation, we suppose as a convention throughout the paper that the universe $G$ is always a finitely generated group. 
\par 
Following \cite[Corollaire 6.4.2]{grothendieck-ega-1-1}, an algebraic variety over an algebraically closed field $K$  is a reduced $K$-scheme of finite type and is 
identified with the set  $X(K)$ of $K$-points. An algebraic group is a group that is an algebraic variety with group operations
given by algebraic morphisms (cf.~\cite{milne}).  Algebraic subvarieties are Zariski
closed subsets and algebraic subgroups are  subgroups which are
also algebraic subvarieties. 

\subsection{Sofic groups} 
The important class of sofic groups was introduced by Gromov \cite{gromov-esav} and Weiss \cite{weiss-sgds} as a common generalization of residually finite groups and  amenable groups. 
Many conjectures for groups have been established for the sofic ones such as 
Gottschalk's surjunctivity conjecture and its  
dual surjunctivity conjecture  (cf.~\cite{kari-post-surjective}). See also 
\cite{capraro}, \cite{ca-and-groups-springer} for some more details. 
 \par
Let $S$ be a finite set. Then an \emph{$S$-label graph} is a pair $\GG= (V,E)$, 
where $V$ is the set of \emph{vertices}, 
and $E \subset V \times S \times V$ is the set of \emph{edges}. 
\par
Denote by $l(\rho)$ 
the \emph{length} of a path $\rho$ in $\GG$. 
If $v, v'\in V$ are not connected by a path in $\GG$, we set $d_\GG(v,v')=\infty$. 
Otherwise, we define  $d_\GG(v,v')\coloneqq \min \{ l(\rho): \text{$\rho$ is a path from $v$ to $v'$}\}$. 
For $v\in V$ and $r \geq 0$,  we define  
\[
B_\GG(v,r)\coloneqq \{v'\in V: d_\GG (v,v') \leq r\}. 
\]
\par 
Observe that $B_\GG(v,r)$ is naturally a finite $S$-labeled subgraph of $\GG$. 
\par
Let $(V_1,E_1)$ and $(V_2,E_2)$ be two $S$-label graphs. 
A map $\phi \colon V_1 \to V_2$ is called an \emph{$S$-labeled graph homomorphism} 
from $(V_1, E_1)$ to $(V_2,E_2)$ if $(\phi(v),s, \phi(v')) \in E_2$ for all $(v,s,v')\in E_1$. 
A bijective $S$-labeled graph homomorphism $\phi \colon V_1 \to V_2$ is an 
\emph{$S$-labeled graph isomorphism} if its inverse  $\phi^{-1}\colon V_2 \to V_1$ 
is an  $S$-labeled graph homomorphism. 
\par
Now let $G$ be a finitely generated group and let $S\subset G$ 
be a finite \emph{symmetric} generating subset, i.e., $S=S^{-1}$. 
The \emph{Cayley graph of $G$ with respect to $S$} is the connected $S$-labeled graph $C_S(G) = (V,E)$,  
where $V = G$ and $E=\{(g,s,gs): g\in G \text{ and } s \in S)\}$.  
For $g \in G$ and $r\geq 0$, we denote 
\[
B_S(r)\coloneqq B_{C_S(G)}(1_G,r)  
\]  
\par 
We can characterize sofic groups as follows (\cite[Theorem~7.7.1]{ca-and-groups-springer}).

\begin{theorem}
\label{t:sofic-character}
Let $G$ be a finitely generated group. 
Let $S\subset G$ be a finite symmetric generating subset. 
Then the following are equivalent: 
\begin{enumerate} [\rm (a)]
\item
the group $G$ is sofic;
\item
for all $r, \varepsilon >0$, there exists a finite $S$-labeled graph $\GG=(V,E) $ 
satisfying 
\begin{equation} 
\label{e:sofic-Q(3r)}
\vert V(r) \vert \geq (1 - \varepsilon) \vert V \vert,
\end{equation}
where $V(r)\subset V$ consists of $v \in V$ such that there exists a (unique) $S$-labeled graph
isomorphism $\psi_{v,r} \colon B_S(r) \to B_\GG(v,r)$ with $\psi_{v,r}(1_G) = v$. 
\end{enumerate}
\end{theorem}

Let $0 \leq r\leq r'$. Then $V(r') \subset V(r)$ since every $S$-labeled graph isomorphism $\psi_{v,r'} \colon B_S(r') \to B_\GG(v,r')$ induces by restriction  
an $S$-labeled graph isomorphism $B_S(r) \to B_\GG(v,r)$. 
\par 
We shall need the following well-known Packing lemma (cf. \cite{weiss-sgds},  \cite[Lemma~7.7.2]{ca-and-groups-springer}, see also \cite{phung-2020} for (ii)).  

\begin{lemma}
\label{l:sofic-V'} 
With the notation as in Theorem~\ref{t:sofic-character}, the following hold 
\begin{enumerate}[\rm(i)]
\item
$B_\GG(v,r) \subset V(kr)$ for all $v \in V((k+1)r)$ and $k \geq 0$;
\item
There exists a finite subset $V' \subset V(3r)$ such that the balls $B_\GG(q,r)$ 
are pairwise disjoint for all $v\in V'$ and that 
$V(3r) \subset \bigcup_{v\in V'} B_\GG(v,2r)$. 
\end{enumerate}
\end{lemma}

\subsection{Tilings of groups} 
Let $G$ be a group and let $E,E' \subset G$. A subset $T \subset G$ is called  an $(E,E')$-\emph{tiling} if:
\begin{enumerate}[\rm (T-1)]
\item 
 the subsets $gE$, $g \in T$, are pairwise disjoint,
 \item
$G = \bigcup_{g \in T} gE'$.
\end{enumerate}
 \par
We shall need the following existence  result which is an immediate consequence of Zorn's lemma 
(see \cite[Proposition~5.6.3]{ca-and-groups-springer}).

\begin{proposition} 
\label{p:tilings-exist}
Let $G$ be a group. Let $E$ be a non-empty finite subset of $G$ and let
$E' \coloneqq  EE^{-1} = \{g h^{-1} : g,h \in E\}$. 
Then there exists  an $(E,E')$-tiling $T \subset G$. \qed 
\end{proposition}

\subsection{Amenable group and algebraic mean dimension} 
Amenable groups were introduced by von Neumann in \cite{neumann-amenable}. 
A group $G$ is \emph{amenable} if it admits a \emph{F\o lner net}, i.e., a family $(F_i)_{i \in I}$ over a directed set $I$ consisting of nonempty finite subsets of $G$ such that
\begin{equation}
\label{e:folner-s}
\lim_{i \in I} \frac{\vert F_i \setminus  F_i g\vert}{\vert F_i \vert} = 0 \text{  for all } g \in G. 
\end{equation}

\par 

In \cite{cscp-alg-goe}, algebraic mean dimension is introduced as an analogue of topological and measure-theoretic entropy, as well as  various notions of mean dimension studied by Gromov in~\cite{gromov-esav}.

\begin{definition}
\label{def:mean-dim}
Let $G$ be an amenable group and let $\FF=(F_i)_{i \in I}$ be a F\o lner net for $G$.  
Let $X$ be an algebraic variety over an algebraically closed field $K$ and 
let $A \coloneqq X(K)$. 
The \emph{algebraic mean dimension}  of a subset
$\Gamma \subset A^G$ with respect to  $\FF$ 
is the quantity $\mdim_\FF(\Gamma)$
defined by
\begin{equation}  
\label{e;mdim}
\mdim_\FF(\Gamma) \coloneqq  \limsup_{i \in I} \frac{\dim(\Gamma_{F_i})}{| F_i |},
\end{equation}
where $\dim(\Gamma_{F_i})$ denotes the Krull dimension of $\Gamma_{F_i}= \{x\vert_{F_i}\colon x \in \Gamma\} \subset A^{F_i}$ with respect to the Zariski topology
and $|\cdot|$ denotes cardinality.
\end{definition}

We shall need the following technical lemma in Section~\ref{s:myhill} and Section~\ref{s:moore}. 
\begin{lemma}
\label{l:inv-mdim-non-max}
Let $G$ be an amenable group and let $\FF=(F_i)_{i \in I}$ be a F\o lner net for $G$.  
Let $X$ be an algebraic variety over an algebraically closed field $K$ and 
let $A\coloneqq X(K)$. 
Suppose that  $\Gamma \subset A^G$ satisfies the following condition:
\begin{enumerate}[{\rm (C)}]
\item
there exist finite subsets $E,E' \subset G$ and an $(E,E')$-tiling $T \subset G$ such that for all $g \in T$,  
$\Gamma_{gE} \subsetneqq A^{gE}$ 
is a proper closed subset of $A^{gE}$ for the Zariski topology.  
\end{enumerate} 
Then one has $\mdim_\FF(\Gamma) < \dim(X)$.
\end{lemma}

\begin{proof}
See \cite[Lemma~5.2]{cscp-alg-goe}. 
\end{proof}

\subsection{Strongly irreducible subshifts}
\label{s:alg-ca}
Let $G$ be a group and let $A$ be a set. A   \emph{subshift} of $A^G$ is a $G$-invariant subset. A subshift of $A^G$ is called a \emph{closed subshift} if it is closed in $A^G$ with respect to the prodiscrete topology. 
\par 
We say that a subshift $\Sigma \subset A^G$ is \emph{strongly irreducible} if there exists a finite subset $\Delta \subset G$ such that for all finite subsets $E , F \subset G$ with $E \Delta \cap F= \varnothing$ and $x, y \in \Sigma$, there exists $z \in \Sigma$ such that $z\vert_{E}= x\vert_{E}$ and $z\vert_{F}=y\vert_{F}$. 


\subsection{Algebraic subshifts and algebraic cellular automata}  

Let $G$ be a group. Let $X$ be an algebraic variety over an 
algebraically closed field $K$. 
\par 
For every finite subset $D \subset G$ and algebraic subvariety $W \subset A^D$, the set 
$$
\Sigma(A^G; W, D)= \{ x \in A^G\colon (gx)\vert_{D}\in W, \mbox{ for all }g \in G \}
$$
is a closed subshift of $A^G$ and 
is called an \emph{algebraic subshift of finite type}.
\par
Following \cite{cscp-alg-ca}, the set $CA_{alg}(G,X,K)$ of \emph{algebraic cellular automata} consists of cellular automata $\tau \colon A^G \to A^G$ which  admit a memory set $M$ 
with local defining map $\mu \colon A^M \to A$ induced by some $K$-morphism of algebraic varieties 
$f \colon X^M \to X$, i.e., $\mu=f\vert_{A^M}$. 
\par 
Let $\Lambda \subset A^G$ be a subshift. If 
 $\Lambda= \tau(\Sigma)$ for some  $\tau \in CA_{alg}(G,X,K)$ and an algebraic subshift of finite type $\Sigma \subset A^G$ then we call $\Lambda$ an \emph{algebraic sofic subshift} (cf.~\cite{cscp-invariant-ca-alg}). See also  \cite{cscp-jpaa}, \cite{phung-dcds} for the simpler linear case.  

\section{Induced maps on the set of connected components}
\label{s:induced-map-connected-component}

Let us fix an algebraically closed field $K$. 
For every $K$-algebraic variety $U$, we denote by $U_0$ the finite set of connected components of $U$ 
and let $i_U \colon U \to U_0$ be the map 
sending every point $u \in U$ to the connected component of $U$ which contains $u$. It is clear that $(U^n)_0= (U_0)^n$ for every $n\in \N$. 
\par 
If $\pi \colon R \to T$ is a morphism of $K$-algebraic varieties, we denote also by $\pi_0 \colon R_0 \to T_0$ the induced map which sends every $p \in R_0$ to $q_0 \in T$ where $q_0$ is the connected component of $T$ containing $\pi(u)$ for any point $u \in R$ that belongs to $p_0$. 
\par 
Remark that $\pi_0$ is well-defined since the image of every connected component of $R$ under $\pi$ is connected. Moreover, it is immediate from the definition that 
\begin{equation}
\label{e:symbolic-map-functorial}
    i_T \circ \pi = \pi_0 \circ i_R.
\end{equation}
 If in addition the map $\pi$ is surjective then clearly $\vert T_0 \vert  \leq \vert R_0 \vert $. 
\par 
Now let $X$ be a $K$-algebraic variety.  
Suppose that $G$ is a group and that $\tau \colon X(K)^G \to X(K)^G$ is an algebraic cellular automaton with an algebraic local defining map $f\colon X^M \to X$ for some finite symmetric subset $M \subset G$, i.e., $M=M^{-1}$. 
Then we obtain a well-defined cellular automaton $\tau_0 \colon X_0^G \to X_0^G$ admitting $f_0 \colon X_0^M \to X_0$ as a local defining map: 
\begin{equation} 
    \tau_0(c)(g)= f_0((g^{-1}c)\vert_M)
\end{equation} 
for all $c \in X_0^G$ and $g \in G$. Let $i_{X^G} \colon X^G \to X_0^G$ be the induced map $i_{X^G}= \prod_G i_X$. Then it is clear that 
\begin{equation}
\label{e:symbolic-induced-map-functorial-ca-alg-1} 
        i_{X^G} \circ \tau = \tau_0 \circ i_{X^G}. 
\end{equation}

\par
Moreover, we infer from the relation \eqref{e:symbolic-map-functorial} the functoriality of our construction of induced cellular automata: for all $\tau, \sigma \in CA_{alg}(G,X,K)$, we have 
\begin{equation}
    \label{e:functorial-induced-symbolic-ca-alg} 
    (\sigma \circ \tau)_0 = \sigma_0 \circ \tau_0. 
\end{equation}
\par 
Indeed, let $f\colon X^M \to X$ and let $h \colon X^M \to X$ be respectively the algebraic local defining maps of $\tau$ and $\sigma$ for some finite memory set $M \in G$. Let $f^+_M \colon X^{M^2} \to X^M$ be the induced map given by 
$f^+_M(c)(g)= f((g^{-1}c)\vert_M)$ for every $c \in A^{M^2}$ and $g \in M$.. Then $h \circ f^+_M\colon X^{M^2} \to X$ is an algebraic local defining map of $\sigma \circ \tau$ associating with the memory set $M^2$. Since 
$$
(h \circ f^+_M)_0 = h_0 \circ (f^+_M)_0, 
$$
we deduce without difficulty that $(\sigma \circ \tau)_0 = \sigma_0 \circ \tau_0$.

\begin{lemma}
\label{l:induced-map-homo-1} Let $\pi \colon R \to T$ be a homomorphism of algebraic groups over $K$. 
Then the induced map $\pi_0 \colon R_0 \to T_0$ is naturally a homomorphism of groups such that $i_T \circ \pi = \pi_0 \circ i_R$. 
\end{lemma} 

\begin{proof}
First, we check that $R_0$ and $T_0$ inherit naturally a group structure from $R$ and $T$ respectively. Consider for example the multiplication 
homomorphism $m \colon R \times R \to R$. Then the multiplication homomorphism $m_0 \colon R_0 \times R_0 \to R_0$ is defined by sending $(p,q) \in R_0 \times R_0$ to $r \in R_0$ where $r$ is the connected component of $R$ which contains the point $xy=m(x,y)$ for any $x \in p$ and $y \in q$. The inverse map $\iota_0\colon R_0 \to R_0$ is defined similarly. Let $\varepsilon \in R_0$ be the connected component which contains $e_R$. Then it is not hard to see that with $\varepsilon$ as the neutral element, $R_0$ is naturally a group with the multiplication map $m_0$ and the inverse map $\iota_0$. The group structure of $T_0$ is defined similarly. 
\par 
Therefore, it follows immediately from the relation~\eqref{e:symbolic-map-functorial} and the definition of the group structures of $R_0$ and $T_0$ that $\pi_0$ is a group homomorphism. 
\end{proof}

\begin{lemma}
\label{l:symbolic-map-conn-also-algr} 
Let $G$ be a group and 
let $X$ be an algebraic group over $K$. 
Suppose that  $\tau \in CA_{algr}(G,X,K)$. Then the induced cellular automaton $\tau_0 \colon X_0^G \to X_0^G$ is a group cellular automaton. 
\end{lemma}

\begin{proof}
Since $\tau \in  CA_{algr}(G,X,K)$, it admits an algebraic local defining map $f\colon X^M \to X$ for some finite subset $M \subset G$. Then the induced map $f_0 \colon X_0^M \to X_0$ is a local defining map of the cellular automaton $\tau_0 \colon X_0^G \to X_0^G$. By Lemma~\ref{l:induced-map-homo-1}, $f_0$ is a homomorphism of groups and we deduce that 
 $\tau_0$ is a group cellular automaton. 
\end{proof}

\section{Weak pre-injectivity} 
\label{s:weak-pre-injectivity}

We recall the following two notions of weak pre-injectivity introduced in \cite[Definition~8.1]{phung-2020}. 

\begin{definition}
\label{d:weak-pre}
Let $G$ be a group. 
Let $X$ be a $K$-algebraic group with neutral element $e$ and let $A=X(K)$. 
Let $\tau \in CA_{algr}(G,X,K)$. 
If $D\subset A^\Omega$ for some finite subset $\Omega \subset G$, we write
\[ 
D_e\coloneqq D \times \{e \}^{G \setminus \Omega} \subset A^G. 
\] 
\begin{enumerate}  [\rm (a)]
\item 
$\tau$ is \emph{$(\sbullet[.9])$-pre-injective} if
there do not exist a finite subset $\Omega \subset G$
 and a Zariski closed subset $H \subsetneq A^\Omega $ such that
 \[
 \tau((A^\Omega)_e)=\tau(H_e). 
 \]
\item 
$\tau$ is \emph{$( \sbullet[.9] \sbullet[.9]) $-pre-injective} if
for every finite subset $\Omega \subset G$, we have 
 \[
 \dim(\tau((A^\Omega)_e)) = \dim(A^\Omega). 
 \]
 \end{enumerate} 
\end{definition}

We establish first the following lemma. 

\begin{lemma}
\label{l:moore-lamma-ca-algr-full-1} 
Let $f\colon X \to Y$ be a  homomorphism of algebraic groups over a field $K$.  
Suppose that $\dim X > \dim f(X)$.  
Then there exists a closed subset $ Z \subsetneq X$ such that $\dim Z < \dim X$ and $f(Z)=f(X)$.  
\end{lemma}

\begin{proof}
Let us write $X= \cup_{i \in I} X_i$ as a disjoint  union of connected components of $X$ where $I$ is a finite set. For each $i \in I$, consider the restriction algebraic morphism $f_i=f\vert_{X_i} \colon X_i  \to f(X_i)$. 
\par 
By milne, we know that the image $f(X)$ is an algebraic group. It follows that $f_i(X_i)$ is a connected component of $f(X)$ for every $i \in I$. 
Note that every connected component of an algebraic group is also an irreducible component and has the same dimension as the dimension of the algebraic group. 
The morphisms $f_i \colon X_i \to f(X_i)$ are surjective morphisms of irreducible algebraic varieties such that $\dim X_i > \dim f(X_i)$.  Hence, \cite[Lemma 8.2]{phung-2020} implies that for every $i \in I$, there exists a proper closed subset $Z_i \subsetneq X_i$ such that $f_i(Z_i) = f(X_i)$. In particular, since $X_i$ is irreducible, it follows that $\dim Z_i < \dim X_i$ for every $i \in I$. 
\par 
Let $Z= \cup_{i \in I} Z_i \subset X$ then we find by construction that 
$$
f(Z)= \cup_{i \in I} f(Z_i) = \cup_{i \in I} f(X_i)= f(X) 
$$
and clearly 
$$\dim Z= \max_{i \in I} \dim Z_i <\max_{i \in I} X_i = \dim X. 
$$
\par 
Therefore, $Z$ verifies the desired properties and the proof of the lemma is thus complete. 
\end{proof} 
\par 
Lemma~\ref{l:moore-lamma-ca-algr-full-1} allows us to show the following general logical implication in the class $CA_{algr}$: 
\[
( \sbullet[.9])\mbox{-pre-injectivity} \implies (\sbullet[.9] \sbullet[.9])\mbox{-pre-injectivity}.
\]
\begin{corollary} 
\label{c:moore-one-star-implies-two-star-pre-injectivity}
Let $G$ be a group and let $X$ be an algebraic group over $K$. 
Let $\tau\in CA_{algr}(G,X,K)$. Suppose that $\tau$ is 
$( \sbullet[.9]) $-pre-injective. Then $\tau$ is also 
$( \sbullet[.9] \sbullet[.9]) $-pre-injective. 
\end{corollary}

\begin{proof}
Suppose on the contrary that $\tau$ is not $( \sbullet[.9] \sbullet[.9]) $-pre-injective. Then we can find a finite subset $E \subset G$ such that $\dim \tau((A^E)_e) <  \dim A^E$. Hence, we infer from Lemma~\ref{l:moore-lamma-ca-algr-full-1}  that there exists a closed subset $Z \subset A^E$ such that $\dim Z < \dim A^E$ and that $\tau((A^E)_e)=\tau(Z_e)$. Since $\dim Z < \dim A^E$, we have $Z \subsetneq A^E$ and we can thus  conclude that $\tau$ is not $( \sbullet[.9]) $-pre-injective, which is a contradiction. The proof is thus complete. 
\end{proof}

\par 
Let $X$ be a connected algebraic group over an algebraically closed field $K$ and let $G$ be a group. 
Let $\tau\in CA_{algr}(G,X,K)$. Then 
it was shown in \cite[Proposition~8.3]{phung-2020} that $\tau$ is {$( \sbullet[.9]) $-pre-injective} if and only if it is 
 {$(\sbullet[.9] \sbullet[.9])$-pre-injective}. 
However, the following result shows that the converse of Corollary~\ref{c:moore-one-star-implies-two-star-pre-injectivity} fails as soon as the alphabet is not a connected algebraic group.  

\begin{proposition}
\label{p:moore-one-star-no-implied-by-two-star-pre-injectivity} 
Let $G$ be a group and let $K$ be an algebraically closed field. Then there exist an algebraic group $X$ over $K$ and $\tau\in CA_{algr}(G,X,K)$ such that $\tau$ is $( \sbullet[.9] \sbullet[.9]) $-pre-injective but is not 
$( \sbullet[.9]) $-pre-injective. 
\end{proposition}

\begin{proof}
Let $X= \Z/4\Z$ and consider the homomorphism $\varphi \colon X \to X$ given by $x \mapsto 2x$. Let  $Y = \Ker \varphi \simeq \Z /2\Z$ then we also have $\varphi(X)=Y$. Let us denote $H=X \setminus \{e\} \subsetneq X$.  
\par 
We define $\tau \colon X^G \to X^G$ by $\tau(c)(g) = \varphi(c(g))$ for all $g \in G$ and $c \in X^G$. 
\par 
Now let $E \subset G$ be a finite subset. Then it is clear by the construction that  we have an equality of Krull dimensions 
$\dim \tau((X^E)_e)= \dim X^E =0$.  It follows that $\tau$ is $( \sbullet[.9] \sbullet[.9]) $-pre-injective. However, we have 
\begin{equation}
\label{e:example-two-star-does-not-implies-one-star}
\tau((X^E)_e)= \tau((H^E)_e)= (Y^E)_e \end{equation}
and $H^E \subsetneq X^E$ is a proper closed subset. Consequently, $\tau$ is not $( \sbullet[.9]) $-pre-injective and the proof is complete. 
\end{proof}

\section{Uniform post-surjectivity} 
\label{s:uniform-post-surj}
We will show in this section that the class $CA_{algr}$ admits a uniform post-surjectivity property (cf. Lemma~\ref{l:uniform-psot-surjectivity}).  
We also prove in Theorem~\ref{t:post-surjective-implies-surjective} that in the class $CA_{alg}$, we have the implication: 
\[
\mbox{post-surjectivity} \implies \mbox{surjectivity}.
\]

\subsection{Post-surjectivity implies surjectivity}
This subsection is independent of the rest of the paper. We begin with the following uniform property of strong irreducibility which is a generalization of \cite[Proposition~1]{kari-post-surjective}. 

\begin{lemma} 
\label{l:strongly-irred-post-surjective-lemma}
Let $G$ be a countable group and let $K$ be an uncountable algebraically closed field.  Let $A$ be an algebraic variety over $K$. 
Let $\Sigma \subset A^G$ be a strongly irreducible closed algebraic subshift. Then there exists a finite subset $\Delta \subset G$ such that for every $x,y \in \Sigma$ and every finite subset $E \subset G$, we can find $z \in \Sigma$ which coincides with $x$ outside of $E \Delta$ and 
$z\vert_E=y\vert_E$. 
\end{lemma} 

\begin{proof} 
Since $\Sigma$ is strongly irreducible, there exists a finite subset $\Delta \subset G$ with $1_G \in \Delta$ such that for every finite subsets $E_1 , E_2 \subset G$ with $E_1 \Delta \cap E_2= \varnothing$ and every $z_1, z_2 \in \Sigma$, there exists $z \in \Sigma$ such that $z\vert_{E_1}= z_1\vert_{E_1}$ and $z\vert_{E_2}=z_2\vert_{E_2}$. 
\par 
Let $(F_n)_{n \in \N}$ be an increasing sequence of finite subsets of $G$ such that $G= \cup_{n \in N} F_n$ and such that $E\Delta \subset F_n$ for every $n \in \N$. 
We set $H_n= F_n \setminus E \Delta$. Then for every $n \in\N$, there exists by the strong irreducibility of $\Sigma$ a configuration $z_n \in \Sigma$ such that $z_n \vert_E= y\vert_E$ and such that  $z_n\vert_{H_n}= x\vert_{H_n}$.  
\par 
Let us define for $n \in \N$: 
$$
\Lambda_n \coloneqq \{ c \in \Sigma_{F_n} \colon c\vert_{E}= y \vert_E, c\vert_{H_n}= x\vert_{H_n}\}. 
$$
\par 
Then by the above paragraph, we deduce that $(\Lambda_n)_{n \in \N}$ forms an inverse system of nonempty algebraic varieties over $K$. The transition maps are simply induced by the restriction maps $A^{F_m} \to A^{F_n}$ for $0 \leq n \leq m$. Hence, by applying \cite[Lemma~B.2]{cscp-alg-ca}, $\varprojlim_{n \in \N} \Lambda_n $ is nonempty and thus we can find 
$$
z \in \varprojlim_{n \in \N} \Lambda_n \subset \varprojlim_{n \in \N} \Sigma_{F_n}= \Sigma.
$$
\par 
The last equality follows from the closedness of $\Sigma$ in $A^G$ with respect to the prodiscrete topology. 
\par 
It is clear from the construction that $z \in \Sigma$ is asymptotic to $x$ and such that 
$z\vert_E=y\vert_E$. In fact, $z$ and $x$  coincide outside of $E\Delta$. The proof is thus complete. 
\end{proof}

We obtain the following generalization of \cite[Proposition~2]{kari-post-surjective}.  

\begin{theorem}
\label{t:post-surjective-implies-surjective}
Let $G$ be a countable group and let $K$ be an uncountable algebraically closed field.  Let $X$ be an algebraic variety over $K$ and let $A=X(K)$. 
Let $\Sigma \subset A^G$ be a strongly irreducible algebraic sofic subshift. Suppose that $\tau \colon \Sigma \to \Sigma$ is the restriction of some $\sigma \in CA_{alg}(G,X,K)$. Then if $\tau$ is post-surjective,  it is also  surjective. 
\end{theorem}

\begin{proof}
Let us fix $x_0, y\in \Sigma$ and a memory set $M \subset G$ of $\tau$. 
Let $(E_n)_{n \in \N}$ be an increasing sequence of finite subsets of $G$ such that $G= \cup_{n \in N} E_n$. 
Then for every $n \in \N$, we infer from  Lemma~\ref{l:strongly-irred-post-surjective-lemma} that there exists $z_n \in \Sigma$ asymptotic to $\tau(x_0)$ and such that $z_n\vert_{E_n} = y\vert_{E_n}$. 
\par
Since $\tau$ is post-surjective and $\tau(x_0) \in \im (\tau)$, it follows 
that $z_n \in \im (\tau)$ for every $n \in \N$. As $(E_n)_{n \in \N}$ is an increasing sequence of finite subsets of $G$ such that $G= \cup_{n \in N} E_n$,  we deduce that $y$ belongs to the closure of $\im (\tau)$ with respect to the prodiscrete topology. 
\par 
Since $\im (\tau)$ is closed by \cite[Theorem~8.1]{cscp-invariant-ca-alg}, it follows that $y \in \im (\tau)$. Therefore, $\tau$ is surjective and the proof is complete. 
\end{proof} 
\par 
We remark here that the exact same proof shows that Theorem~\ref{t:post-surjective-implies-surjective} still holds  if $X$ is an algebraic group over an arbitrary algebraically closed field $K$, $\Sigma \subset A^G$ is a strongly irreducible algebraic group subshift (cf.~\cite[Definition~1.2]{phung-israel}), and $\tau\colon \Sigma \to \Sigma$ is the restriction of some $\sigma \in CA_{algr}(G,X,K)$. It suffices to observe that in this situation, $\im (\tau)$ is still closed in $A^G$ thanks to  \cite[Theorem~4.4]{phung-israel}.

\subsection{Uniform post-surjectivity}  

We have the following key uniform property for the post-surjectivity in the class $CA_{algr}$. 

\begin{lemma}[Uniform post-surjectivity] 
\label{l:uniform-psot-surjectivity}
Let $G$ be a countable group.  Let $X$ be an algebraic group over an uncountable algebraically closed field $K$ and suppose that $\tau \in CA_{algr}(G,X,K)$ is post-surjective. Let $A=X(K)$. Then there exists a finite subset $E\subset G$ with the following property. For all $x, y\in A^G$ such that $y\vert_{G \setminus \{1_G\}} =\tau(x)\vert_{G \setminus \{1_G\}}$, there exists $x' \in A^G$ such that $\tau(x')=y$ and 
$x'\vert_{G \setminus E}= x\vert_{G \setminus E}$. 
\end{lemma}

\begin{proof}
Let $M \subset G$ be a finite memory set of $\tau$ such that $1_G \in M$ and $M=M^{-1}$. Let $\mu \colon A^M \to A$ be the corresponding local defining map. 
\par 
Since $\tau \in CA_{algr}(G,X,K)$, it follows that $\mu$ is induced by a homomorphism of algebraic groups $f \colon X^M \to X$. 
\par 
Let $(E_n)_{n \in \N}$ be an exhaustion of $G$ consisting of finite groups such that $1_G \in E_0$. For each $n \in \N$, we define 
\begin{equation}
\label{e:uniform-post-surjectivity-eq-1}
V_n \coloneqq 
\{ x\in A^G \colon \tau(x)\vert_{G \setminus \{1_G\}} = e^{G \setminus \{1_G\}}, x\vert_{G \setminus E_n} = e^{G \setminus E_n} \}. 
\end{equation}
\par 
Consider the following homomorphism of algebraic groups 
\begin{align}
    \varphi_n \colon A^{E_n}\times \{e\}^{E_n M^2 \setminus E_n} \to A^{E_nM}
\end{align}
defined by $\varphi_n(x)(g) = \mu((g^{-1}x)\vert_M)$ for all $x \in A^{E_n}\times \{e\}^{E_n M^2 \setminus E_n}$ and $g \in E_nM$. 
\par 
Note that $1_G \in E_n M$ for every $n\in \N$. We denote respectively by $p_n \colon  A^{E_nM} \to A^{\{1_G\}}$ and $q_n \colon A^{E_nM} \to A^{E_nM\setminus \{1_G\}}$ the canonical projections. 
\par
It is clear that for all $n \in \N$, we can identify 
\begin{equation}
  V_n=\Ker q_n \circ \varphi_n   
\end{equation}
which is  an algebraic subgroup of $A^{E_n}\times \{e\}^{E_n M^2 \setminus E_n}=A^{E_n}$. 
Let us consider  
\begin{equation}
    Z_n \coloneqq p_n (\varphi_n (V_n))= \tau(V_n)_{\{1_G\}},  \quad T_n \coloneqq A \setminus Z_n. 
\end{equation} 
\par 
Then $Z_n$ is an algebraic subgroup of $A$ for every $n \in N$. Since $E_n \subset E_{n+1}$ for every $n \in \N$, we deduce from \eqref{e:uniform-post-surjectivity-eq-1} that $V_n \subset V_{n+1}$. Consequently, we find that 
$Z_n \subset Z_{n+1}$ for all $n \in \N$. 
\par 
We claim that $\cup_{n \in \N}Z_n=A$. 
Indeed, let $y \in A$ and consider $c \in A^G$ defined by $c(g)=e$ for all $g \in G\setminus \{1_G\}$ and $c(1_G)=y$.  Since $\tau$ is post-surjective and since $\tau(e^G)=e^G$, it follows that there exist $x \in A^G$ and $n \in \N$ such that $x\vert_{G \setminus E_n}=e^{G \setminus E_n}$ and such that $\tau(x)=c$. We deduce that $\tau(x)\vert_{G\setminus \{1_G\}}=e^{G\setminus \{1_G\}}$ and thus $x \in V_n$. Moreover, as $\tau(x)(1_G)=y$, it follows that $y \in Z_n$. Hence, we have proven the claim 
that $\cup_{n \in \N}Z_n=A$. 
\par 
Therefore, $(T_n)_{n \in \N}$ is a decreasing sequence of constructible subsets of $A$ such that 
$$
\cap_{n \in \N} T_n=A\setminus (\cup_{n \in \N} Z_n)= \varnothing. 
$$
\par 
Since the field $K$ is uncountable and algebraically closed, we infer from \cite[Lemma~B.2]{cscp-alg-ca} (see also  \cite[Lemma~3.2]{cscp-invariant-ca-alg}) that there exists $N \in \N$ such that $T_N=\varnothing$. It follows that $Z_N=A$. We claim that $E \coloneqq E_N$ satisfies the desired property in the conclusion of the lemma. 
\par 
Indeed, suppose that $x , y \in A^G$ satisfy $y\vert_{G \setminus \{1_G\}} =\tau(x)\vert_{G \setminus \{1_G\}}$. Then, let $c = y (\tau(x))^{-1} \in A^G$ then $c\vert_{G\setminus \{1_G\}}=e^{G\setminus \{1_G\}}$. Since $Z_N=\tau(V_N)_{\{1_G\}}= A$, we can find 
$d \in V_N$ such that $\tau(d)(1_G)=c(1_G)$. 
Therefore, $d\vert_{G\setminus E_N}=e^{G\setminus E_N}$ and $\tau(d)\vert_{G\setminus \{1_G\}}= e^{G\setminus \{1_G\}}$. 
\par 
Consequently, since $\tau$ is a homomorphism, we find for $x'\coloneqq dx \in A^G$ and for every $g \in G$ that 
\begin{align*} 
\tau(x')(g) & =  \tau (d)(g) \tau(x)(g) \\
&= 
  \begin{cases}
		\tau(x)(g)   & \mbox{if } g \in G\setminus \{1_G\} \\
		y(1_G) (\tau(x)(1_G))^{-1} \tau(x)(1_G)  & \mbox{if } g=1_G
	\end{cases}
\\ 
&= 
  \begin{cases}
		 y(g)  & \mbox{if } g \in G\setminus \{1_G\} \\
		y(1_G) & \mbox{if } g=1_G
	\end{cases}
\\ 
& = y(g). 
\end{align*}
\par 
Therefore, $\tau(x')=y$. On the other hand, since  $d\vert_{G\setminus E}=e^{G\setminus E}$ and $x'=dx$, we have  $x'\vert_{G \setminus E}= x\vert_{G \setminus E}$. The conclusion thus follows. 
\end{proof}

\section{Dual surjunctivity for $CA_{algr}$} 
\label{s:post-surjectivity}

In this section, we will present the proof of Theorem~\ref{t:post-surj-full-ca-algr} and the construction of Example~\ref{ex:post-sur-not-pre-inj} showing that in a certain sense,  Theorem~\ref{t:post-surj-full-ca-algr} is optimal. 

\begin{theorem}
\label{t:dual-gottschalk-ca-algr-full}
Let $G$ be a sofic group and let $X$ be an algebraic group over an uncountable algebraically closed field $K$. 
Suppose that $\tau \in CA_{algr}(G,X,K)$ is post-surjective. Then $\tau$ is both  $(\sbullet[.9])$-pre-injective and $(\sbullet[.9] \sbullet[.9])$-pre-injective. 
\end{theorem}

\begin{proof} 
Let $A=X(K)$ and let $S \subset G$ be a memory set of $\tau$ such that $1_G \in S$, $S=S^{-1}$ and that $S$ generates $G$. Let $f \colon A^S \to S$ be the corresponding local defining map which is a homomorphism of algebraic groups. 
\par 
Since $(\sbullet[.9])$-pre-injectivity implies $(\sbullet[.9] \sbullet[.9])$-pre-injectivity in the class $CA_{algr}$ (cf. Lemma~\ref{c:moore-one-star-implies-two-star-pre-injectivity}), it suffices to show that $\tau$ is  $(\sbullet[.9])$-pre-injective. 
\par 
Suppose on the contrary that $\tau$ is not 
$(\sbullet[.9])$-pre-injective. Then there exist a finite subset $\Omega\subset G$ and a proper closed subset $H \subsetneq A^\Omega$ such that 
\begin{equation} 
\label{e:dual-post-surjecitve-weak-pre-injective-1}  
\tau((A^\Omega)_e)=\tau(H_e). 
\end{equation}

Since $\tau$ is post-surjective, there exists a finite subset $E \subset G$ with the property described in Lemma~\ref{l:uniform-psot-surjectivity}, i.e., for $x, y\in A^G$ with $y\vert_{G \setminus \{1_G\}} =\tau(x)\vert_{G \setminus \{1_G\}}$, there exists $x' \in A^G$ such that $\tau(x')=y$ and 
$x'\vert_{G \setminus E}= x\vert_{G \setminus E}$. 
\par 
Up to enlarging $E$, we can suppose without loss of generality that for some $r\geq 2$ large enough, we have 
$$
\Omega \subset B_S(r-1) \subset E= B_S(r). 
$$
\par 
    If $\dim A=0$ then $A$ is a finite group so  Theorem~\ref{t:kari-dual-surj} implies that $\tau$ is pre-injective. By \cite[Proposition~6.4.1,  Example~8.1]{cscp-alg-goe} and \cite[Proposition~8.3.(ii)]{phung-2020}, pre-injectivity is equivalent to $(\sbullet[.9])$-pre-injectivity for finite group alphabets. Thus  we deduce that $\tau$ is $(\sbullet[.9])$-pre-injective.  
\par 
Suppose from now on that $\dim A >0$. 
Let $X_0$ be the set of connected components of $X$. 
Let us fix $0<\varepsilon< \frac{1}{2}$ small enough
  so that 
  \begin{equation}
  \label{e:post-surjective-var-epsilon-3-1}
    \vert X_0 \vert^{\varepsilon} \left( 1 - \vert X_0 \vert^{-|B_S(r)|} \right)^{\frac{1}{2|B_S(2r)|}}
    < 1,
  \end{equation}
  and that 
\begin{equation}  
\label{e:post-surjective-var-epsilon-2}
   0<   (1 - \varepsilon)^{-1} < 1 + \frac{1}{  \vert B_S(2r) \vert \dim A}.
\end{equation}
\par 
Since the group $G$ is sofic, it follows from Theorem~\ref{t:sofic-character} 
that there exists a finite $S$-labeled graph $\GG=(V,E)$ associated to the pair $(3r, \varepsilon)$ 
such that
\begin{equation} 
\label{e:sofic-V-2post-surjective}
    \vert V(3r) \vert \geq (1 - \varepsilon) \vert V \vert > \frac{1}{2} \vert V \vert, 
\end{equation}
where for each $s \geq 0$, the subset $V(s)\subset V$ 
consists of $v \in V$ such that there exists a unique $S$-labeled graph
isomorphism $\psi_{v,s}\colon  B_\GG(v,s) \to B_S(s)$ sending  
 $v$ to $1_G$ (cf. Theorem \ref{t:sofic-character}).
Note that $V(s) \subset V(s')$ for all $0 \leq s' \leq s$.  
\par
We denote $B(v,s)\coloneqq B_\GG(v,s)$ for $v \in V$ and $s \geq 0$.  
Define $V' \subset V(3r) $ as in Lemma~\ref{l:sofic-V'}.(ii) so that 
$B(v,r)$ are pairwise disjoint for all $v\in V'$ and that $V(3r) \subset \bigcup_{v \in V'}B(v,2r)$. In particular,  
\begin{equation}
\label{e:v'-and-v-post-surjective-3}
    \vert V(3r) \vert \leq \vert V' \vert \vert B_S(2r) \vert.
\end{equation}
\par 
Let us denote 
$\overline{V'} \coloneqq \coprod_{v \in V'} B(v,r)$. 
Note that the local map $f$ induces a homomorphism of algebraic groups $\Phi \colon A^V \to A^{V(3r)}$  given by 
$\Phi(x)(v)= f(\psi_{v,r}(x\vert_{B(v,r)}))$ for all $x \in A^V$ and $v \in V(3r)$. 
\par 
  Since $E=S=B(r)$, we deduce by applying repeatedly  Lemma~\ref{l:uniform-psot-surjectivity} that the homomorphism  $\Phi$ is surjective (cf. the proof of \cite[Lemma 2]{kari-post-surjective}), that is, 
  \begin{equation} 
  \label{e:surj-post-phi-surjective-v-3r}
  \Phi(A^V)= A^{V(3r)}.
  \end{equation}
  \par 
  We claim that $\dim \Ker \tau\vert_{(A^\Omega)_e}=0$. Indeed, suppose on the contrary that 
  \begin{equation}
      \label{e:main-dual-post-surjective-2} 
      \dim \Ker \tau\vert_{(A^\Omega)_e} \geq 1. 
  \end{equation}
  \par 
  For $s \geq 0$ and  $v \in V(s)$, we denote by $\varphi_{v,s} \colon A^{B_S(s)} \to A^{B(v,s)}$ and $\varphi_{v,s,\Omega}\colon A^{\Omega}  \to A^{\psi_{v,s}(\Omega)}$ the isomorphisms induced respectively by the bijections $\psi_{v,s}$ and  $\psi_{v,s}\vert_\Omega$. 
  \par 
  Since $\Omega \subset B_S(r-1)$,  we can regard $   \Ker \tau\vert_{(A^\Omega)_e}$ as a closed subgroup of $A^{B_S(r-1)} \times \{e\}^{B_S(r)\setminus B_S(r-1)}$. 
  \par 
  As the homomorphism $\Phi$ is naturally induced by the local defining map $f\colon A^S \to A$ of $\tau$ and as the balls $B_S(r)$, $v \in V'$, are disjoint, we deduce that 
 \begin{equation*}
 \{e\}^{V \setminus \overline{V'}}\times \prod_{v \in V'} \varphi_{v, r} (\Ker \tau\vert_{(A^\Omega)_e}) \subset \Ker \Phi. 
 \end{equation*}
 \par 
 Consequently,  the relation \eqref{e:main-dual-post-surjective-2} implies that 
 \begin{align}
 \label{e:dual-surjunctivity-proof-dim-ker-phi-1}
     \dim \Ker \Phi 
     & \geq \sum_{v \in V'} \dim \varphi_{v, r} (\Ker \tau\vert_{(A^\Omega)_e}) \\
     & = \sum_{v \in V'} \dim (\Ker \tau\vert_{(A^\Omega)_e}) \nonumber
     \\
     & \geq \vert V' \vert. \nonumber
 \end{align}
 \par 
Therefore, the Fiber dimension theorem (see e.g.~\cite[Proposition~5.23]{milne}) implies that: 
 \begin{align*}
    \dim \Phi(A^V) & = \dim A^V -\dim \Ker \Phi 
     &  
     \\ & 
     \leq 
      \vert V \vert  \dim A- \vert V' \vert
       & (\text{by } \eqref{e:dual-surjunctivity-proof-dim-ker-phi-1})
    \\ & \leq 
      (1-\varepsilon)^{-1} \vert V(3r) \vert \dim A - \frac{\vert V(3r) \vert}{\vert B_S(2r)\vert}
       & (\text{by } \eqref{e:sofic-V-2post-surjective})
    \\ &  
    \leq 
    \vert V(3r) \vert  \dim A  \left( (1-\varepsilon)^{-1} - 
    \frac{1}{\vert B_S(2r)\vert \dim A } \right)
     & (\text{by } \eqref{e:v'-and-v-post-surjective-3})
    \\
    & <   \vert V(3r) \vert \dim A 
    & (\text{by }  \eqref{e:post-surjective-var-epsilon-2})
    \\
    & = \dim A^{V(3r)}. 
\end{align*} 
\par 
However, since $\Phi(A^V)= A^{V(3r)}$ by \eqref{e:surj-post-phi-surjective-v-3r}, we arrive at a contradiction. Thus, we have proven the claim that $\dim \Ker \tau\vert_{(A^\Omega)_e}=0$. 
  \par 
  In what follows, we shall distinguish two cases according to whether $\dim H < \dim A^{B_S(r)}$ or $\dim H = \dim A^{B_S(r)}$.
  \par 
  \underline{\textbf{Case 1}}: $\dim H < \dim A^{\Omega}$. Then we infer from  \eqref{e:dual-post-surjecitve-weak-pre-injective-1} that 
  $$
  \dim \tau((A^\Omega)_e)
  =\dim \tau(H_e) 
  < \dim A^\Omega. 
  $$
  Therefore, the Fiber dimension theorem (cf.~\cite[Proposition~5.23]{milne}) implies that 
  $$
  \dim \Ker \tau\vert_{(A^\Omega)_e} = \dim A^\Omega  - \dim \tau((A^\Omega)_e)\geq 1 
  $$
which is a contradiction since  $ \dim \Ker \tau\vert_{(A^\Omega)_e}=0$. 
   \par 
  \underline{\textbf{Case 2}}: $\dim H = \dim A^{\Omega}$. Since 
  $ \dim \Ker \tau\vert_{(A^\Omega)_e}=0$, it follows from the Fiber dimension theorem (cf.  \cite[Proposition~5.23]{milne}) that 
$$
  \dim \tau(H_e)=\dim \tau((A^\Omega)_e)= \dim A^\Omega.
$$
\par 
We can write $H=H'\cup H''$ where $\dim H'' < \dim H$ and $H'$ is a union of some connected components of $A^\Omega$. It follows that 
$\tau(H_e) = \tau(H'_e)\cup \tau (H''_e)$.
\par 
Note that since $\tau(H_e)= \tau((A^\Omega)_e)$ is an algebraic group, all of its connected components have the same dimension $\dim \tau((A^\Omega)_e)$. 
\par 
On the other hand, since  
$\dim \tau(H''_e) \leq \dim H''< \dim A^\Omega= \dim \tau((A^\Omega)_e)$, we deduce that $\dim  \tau(H'_e)= \dim \tau((A^\Omega)_e)$ and also
\begin{equation}
\label{e:new-h'-post-surjective-1}
\tau((A^\Omega)_e) = \tau(H'_e)\cup \tau (H''_e)= \tau(H'_e).
\end{equation}
\par 
For every algebraic variety $U$, recall that $U_0$ denotes the set of connected components of $U$. Since $\Phi \colon A^V \to A^{V(3r)}$ is surjective, the induced homomorphism $\Phi_0 \colon X_0^V \to X_0^{V(3r)}$ is also surjective. 
\par 
Let $Y \subset X$ denote the neutral  connected component of $X$ and $B= Y(K)$. 
We deduce from \eqref{e:new-h'-post-surjective-1}   that $\tau(H' \times B^{B_S(r)\setminus \Omega})$ has nonempty intersection with every connected component of  $\tau(A^\Omega \times B^{B_S(r)\setminus \Omega})$. In particular, for every $v \in V'$, we find that 
\begin{equation}
    \label{e:new-h'-post-surjective-2}
 \Phi_0((A^{\psi_{v,r}(\Omega)} \times B^{V\setminus \psi_{v,r}(\Omega)})_0)  = 
\Phi_0((\varphi_{v,r,\Omega}(H') \times B^{V\setminus \psi_{v,r}(\Omega)})_0).
\end{equation}
\par 
Note that since $H'$ is a union of some connected components of $A^\Omega$, 
we have $(\varphi_{v,r,\Omega}(H') \times B^{V\setminus \psi_{v,r}(\Omega)})_0 \in X_0^V$. Therefore, in \eqref{e:new-h'-post-surjective-2}, 
the expression  $\Phi_0((\varphi_{v,r,\Omega}(H') \times B^{V\setminus \psi_{v,r}(\Omega)})_0)$ is well-defined. 
\par 
For each $v \in V'$, we consider the following subset of $X_0^{B_S(r)}$: 
\[
I_{v} \coloneqq (X_0^{B_S(r)}\setminus 
(A^{\psi_{v,r}(\Omega)} \times B^{B_S(r)\setminus \psi_{v,r}(\Omega)})_0) \cup 
(\varphi_{v,r,\Omega}(H') \times B^{B_S(r)\setminus \psi_{v,r}(\Omega)})_0. 
\]
\par 
Then since $(H')_0 \subsetneq X_0^{\Omega}$ is a proper subset, we deduce that: 
\begin{equation}
    \label{e:post-surjective-i-v-less-cardinality}
\vert I_v \vert \leq \vert X_0^{B_S(r)} \vert -1. 
\end{equation}
\par 
Moreover, since 
$\overline{V'} \coloneqq \coprod_{v \in V'} B(v,r)$ is a disjoint union of the balls $B(v,r)$ and since $\psi_{v,r}(\Omega) \subset B(v,r-1)$ for all $v \in V'$, 
we infer from \eqref{e:new-h'-post-surjective-2} that 
\begin{align*} 
    \Phi_0((A^V)_0) 
    & = \Phi_0 \left( X_0^{V\setminus \overline{V'}} \times \prod_{v \in V'} 
     I_v \right)  
\end{align*} 
\par 
Taking the cardinality of both sides, we deduce from the relations 
\eqref{e:post-surjective-i-v-less-cardinality}, \eqref{e:v'-and-v-post-surjective-3}, \eqref{e:sofic-V-2post-surjective}, and \eqref{e:post-surjective-var-epsilon-3-1}  that:  

  \begin{align*}
    |\Phi_0(X_0^V)| 
    & \leq  
    \vert   X_0^{V\setminus \overline{V'}} \times \prod_{v \in V'}  I_v \vert  
    & 
    \\
    & \leq   \vert X_0 \vert^{|V|-|V'||B_S(r)|} \left(\vert X_0 \vert^{|B_S(r)|}-1\right)^{|V'|} 
    & (\text{by } \eqref{e:post-surjective-i-v-less-cardinality})
    \\
&   = 
    \vert X_0 \vert^{|V|} \left( 1 - \vert X_0 \vert^{-|B_S(r)|} \right)^{|V'|} 
    & 
     \\ & 
     \leq 
     \vert X_0 \vert^{|V|}
    \left( 1 - \vert X_0 \vert^{-|B_S(r)|} \right)^{\frac{|V(3r)|}{|B_S(2r)|}} 
    & (\text{by } \eqref{e:v'-and-v-post-surjective-3})
    \\ & < 
    \vert X_0 \vert^{|V|}
    \left( 1 - \vert X_0 \vert^{-|B_S(r)|} \right)^{\frac{|V|}{2|B_S(2r)|}} 
    & (\text{by }\eqref{e:sofic-V-2post-surjective})
    \\ &  < 
    \vert X_0 \vert^{|V|}
    \vert X_0 \vert^{-\varepsilon|V|} 
    & (\text{by }\eqref{e:post-surjective-var-epsilon-3-1})
\\ &      = 
    \vert X_0 \vert^{(1-\varepsilon)|V|}
    & 
    \\ & < 
    \vert X_0 \vert^{|V(3r)|}  
    & (\text{by } \eqref{e:sofic-V-2post-surjective})
\end{align*} 
  which is again a contradiction since  $\Phi_0(X_0^V)=X_0^{V(3r)}$. 
   \par 
   Therefore, we can conclude that $\tau$ must be $(\sbullet[.9])$-pre-injective. The proof of the theorem is thus complete. 
\end{proof}

\subsection{A counter-example} 
Using nontrivial covering maps, we present a simple example which shows that in the class $CA_{algr}$, the implication 
$$
\mbox{post-surjectivity} \implies \mbox{ pre-injectivity}
$$
fails over any universe. 

\begin{example}
\label{ex:post-sur-not-pre-inj} 
Let $G$ be a group and let $E$ be a complex elliptic curve with origin $O \in E$. Consider the algebraic group cellular automaton $\tau \colon E^G \to E^G$ defined by $\tau(c)(g)= 2c(g)$ for every $c \in E^G$ and $g \in G$. We claim that $\tau$ is post-surjective but it is not pre-injective. 
\par
Indeed, consider the multiplication-by-2 map $[2] \colon E \to E$, $P \mapsto 2P$. Then $[2]$ is a  covering map of $E$ of degree $4$.  Hence, there exists $P \in I\setminus \{O\}$ such that $2P=O$. Consider $c \in E^G$ given by $c(1_G)=P$ and $c(g)=O$ if $g \in G \setminus \{1_G\}$. It is immediate that $c$ and $O^G$ are asymptotic and distinct but 
$\tau(c)=\tau(O^G)=O^G$. This proves that $\tau$ is not pre-injective. 
\par 
Now let $x, y\in E^G$ such that $y\vert_{G \setminus \Omega}=\tau(x)\vert_{G \setminus \Omega}$ for some finite subset $\Omega \subset G$. Since $[2]$ is surjective, we can find $p \in E^\Omega$ such that $2p(g)=y(g)$ for all $g \in \Omega$. Consider $z \in E^G$ given by $z\vert_{G \setminus \Omega}=x\vert_{G \setminus \Omega}$ and $z\vert_\Omega=p$ then it is clear that $\tau(z)=y$. This shows that $\tau$ is post-surjective. 
\end{example}

\section{Myhill property of $CA_{algr}$} 
\label{s:myhill}

We shall need the following technical result in the proof of Theorem~\ref{t:myhill-ca-algr-full}. 

\begin{proposition}
\label{p:**-implies-mdim}
Let $G$ be an amenable group and let $\FF=(F_i)_{i \in I}$ be a F\o lner net for $G$.  
Let $X$ be an algebraic group over an algebraically closed field $K$ and 
let $A \coloneqq X(K)$. 
Suppose that $\tau \in CA_{alg}(G,X,K)$ is $(\sbullet[.9] \sbullet[.9])$-pre-injective.
Then one has 
\begin{equation}
\label{e:mdim-max}
\mdim_\FF(\tau(A^G)) = \dim(X). 
\end{equation}
\end{proposition}

\begin{proof}
It is a direct consequence of \cite[Proposition~6.5]{cscp-alg-goe}. It suffices to observe there that $CA_{algr} \subset CA_{alg}$ and  in the class $CA_{algr}$, the two notions  $(**)$-pre-injectivity and $(\sbullet[.9] \sbullet[.9])$-pre-injectivity are in fact equivalent by \cite[Proposition~8.3]{phung-2020}. 
\end{proof}

We can now state and prove the Myhill property for the class $CA_{algr}$, which is the content of Theorem~\ref{t:gromov-answer-algr-full}.(i).

\begin{theorem}
\label{t:myhill-ca-algr-full}
Let $G$ be an amenable group and let $X$ be an algebraic group over $K$. Suppose that 
$\tau \in CA_{algr}(G,X,K)$ is pre-injective. Then $\tau$ is surjective. 
\end{theorem}

\begin{proof} 
Let $A=X(K)$ and let $\Gamma=\tau(A^G)$. Then it follows from \cite[Theorem~5.1]{phung-2020} that $\Gamma$ is closed in $A^G$ with respect to the prodiscrete topology. 
\par 
Since $\tau$ is pre-injective, it is   
$(\sbullet[.9] \sbullet[.9])$-pre-injective (cf.~\cite[Proposition~8.3]{phung-2020}). We can thus deduce from  Proposition~\ref{p:**-implies-mdim} that $\mdim_\FF(\Gamma) = \dim(X)$ where $\FF=(F_i)_{i \in I}$ is an arbitrary fixed F\o lner net for $G$. 
\par 
Therefore, it follows immediately from Lemma~\ref{l:inv-mdim-non-max} and Proposition~\ref{p:tilings-exist} that we have an equality of Krull dimensions  $\dim \Gamma_E= \dim X^E$ for every finite subset $E \subset G$. 
\par
On the other hand,  \cite[Theorem~7.1]{cscp-invariant-ca-alg} implies that $\Gamma_E$ is an algebraic subgroup of $A^E$ for every finite subset $E \subset G$. 
\par
Now consider the induced group cellular automaton $\tau_0 \colon X_0^G \to X_0^G$ where the alphabet $X_0$ is the set of connected components of $X$ (cf. Lemma~\ref{l:symbolic-map-conn-also-algr}). We are going to show that $\tau_0$ is also pre-injective. 
\par 
Let $f\colon A^M \to A$, where $M \subset G$ is a finite symmetric subset, be a  homomorphism of algebraic groups which is also a local defining map of $\tau$. 
\par 
Suppose on the contrary that $\tau_0$ is not pre-injective. Consequently, we can find a finite subset $E \subset G$ and subvarieties $V_1, V_2 \subset A^E$ and a subvariety $U \subset A^{ME\setminus E}$ with the following properties: 
\begin{enumerate} [\rm (a)]
    \item $U$ is a connected component of 
$A^{EM\setminus E}$ and $V_1, V_2$ are distinct connected components of $A^E$; 
\item 
the images $\tau^+_E(U \times V_1)$ and  $\tau^+_E(U \times V_2)$ belong to the same connected component $Z$ of the algebraic group $A^E$, where 
the induced homomorphism $\tau^+_E \colon A^{EM} \to A^E$ of algebraic groups is given by 
$\tau^+_E(c)(g)= f((g^{-1}c)\vert_M)$ for every $c \in A^{EM}$ and $g \in E$. 
\end{enumerate}
 \par 
Let us choose an arbitrary point $u \in U$. Then as $\tau$ is pre-injective and as $\dim V_i=\dim Z= \dim A^E$, we must have 
$\tau^+_E (\{u\} \times V_i)=Z$ for $i=1,2$. 
\par 
Indeed, since otherwise, we would have  $\dim \tau^+_E (\{u\} \times V_i) < \dim Z= \dim V_i$. Note that $\{u\} \times V_i$ is an irreducible variety. 
Therefore, by applying  \cite[Proposition~2.11]{cscp-alg-goe}, 
we can find distinct points $s,t \in V_i$  such that $\tau^+_E(u,s)=\tau^+_E(u,t)$. Hence, the map 
$\tau^+_E\vert_{\{u\} \times V_i}$ cannot be injective. It follows that $\tau$ is not pre-injective, which is a contradiction. 
\par 
Therefore, for any $z \in Z$, we can find $v_i \in V_i$ for $i=1,2$ such that $\tau^+_E(u, v_i)=z$. Since $V_1$ and $V_2$ are disjoint, $v_1 \neq v_2$ and  it follows that $\tau$ is not pre-injective which is a contradiction. We conclude that 
$\tau_0$ is indeed pre-injective.
\par 
Hence, since the alphabet $X_0$ is finite and $G$ is an amenable group, we can deduce from the classical Garden of Eden theorem for finite alphabets that $\tau_0$ is surjective.
\par 
Let $E\subset G$ be any finite subset. As $\tau_0$ is surjective, we deduce from the definition of $\tau_0$ that $\Gamma_E$ contains points in every connected component of $A^E$. On the other hand, we have seen that $\Gamma_E$ is an algebraic subgroup of $A^E$ such that $\dim \Gamma_E= \dim A^E$. It follows that $\Gamma_E=X^E$ for every finite subset $E \subset G$. 
\par 
Since the image $\Gamma=\tau(A^G)$ is closed in $A^G$ with respect to the prodiscrete topology, we find that 

$$
\Gamma= \varprojlim_{E \subset G} \Gamma_E= \varprojlim_{E \subset G} A^E= A^G.
$$

\par 
It follows that $\tau$ is surjective and the proof is complete. 
\end{proof}

Our next result shows that in the class $CA_{algr}$, the implication  
$$
(\sbullet[.9] \sbullet[.9])\mbox{-pre-injectivity} \implies \mbox{ surjectivity}
$$

does not hold in any universe $G$.  

\begin{proposition}
\label{p:myhill-ca-algr-two-star-not-hold} 
Let $G$ be a group. Then there exist a finite algebraic group $X$ over $K$ and $\tau \in CA_{algr}(G,X,K)$  such that $\tau$ is 
$(\sbullet[.9] \sbullet[.9])$-pre-injective but is not surjective. 
\end{proposition}

\begin{proof}
Let $X$ and $\tau \in CA_{algr}(G,X,K)$ be given by Proposition~\ref{p:moore-one-star-no-implied-by-two-star-pre-injectivity}. Keep the notations as in the proof of Proposition~\ref{p:moore-one-star-no-implied-by-two-star-pre-injectivity}. Then we know that 
$\tau$ is 
$(\sbullet[.9] \sbullet[.9])$-pre-injective but it is not surjective since 
$\tau(X^G)=Y^G  \subsetneq X^G$. The proof is thus complete. 
\end{proof}

\section{Moore property of $CA_{algr}$} 
\label{s:moore}

To complete the proof of Theorem~\ref{t:gromov-answer-algr-full}, we will prove the following Moore property of the class $CA_{algr}$. 

\begin{theorem}
\label{t:moore-ca-algr-full} 
Let $G$ be an amenable group and let $X$ be an algebraic group over an algebraically closed field $K$. Suppose that $\tau \in CA_{algr}(G,X,K)$  surjective. Then $\tau$ is both  $(\sbullet[.9])$-pre-injective and $(\sbullet[.9] \sbullet[.9])$-pre-injective.  
\end{theorem}

\begin{proof}
Let $A \coloneqq X(K)$ and let  $\FF$ be a F\o lner net for $G$. 
Thanks to Corollary~\ref{c:moore-one-star-implies-two-star-pre-injectivity}, it suffices to show that 
$\tau$ is $(\sbullet[.9])$-pre-injective. For this, 
we shall proceed by contradiction. 
\par 

Suppose that $\tau$ is not $(\sbullet[.9])$-pre-injective. 
Thus, there exist a finite subset $E\subset G$ and a  proper closed subset $H \subsetneq A^E$ such that \begin{equation}
\label{e:dim-hyperplan-ca-algr-full}
\tau((A^E)_e)=\tau(H_e). 
\end{equation}
\par 
We will distinguish two cases according to whether $\dim H = \dim A^E$. 
\par 
\underline{\textbf{Case 1}}: $\dim H < \dim A^E$. By Proposition~\ref{p:tilings-exist}, we can find a finite subset $E' \subset G$
such that $G$ contains an
$(E,E')$-tiling $T$. 
For every $t \in T$, we define $H_t \subset A^{tE}$ to be the image of $H$ under the 
canonical bijective map $A^E \to A^{t E}$ that is induced by the left-multiplication by $t^{-1}$. 
Since $\tau$ is $G$-equivariant, we deduce from \eqref{e:dim-hyperplan-ca-algr-full} that 
for each $t\in T$, we have that 
\begin{equation}
\label{e:dim-hyperplan-g}
\tau((A^{tE})_p)=\tau((H_t)_p) \quad \text{ for all } p \in A^{G\setminus tE}.
\end{equation}
\par
Consider the subset $\Gamma \subset A^G$ defined by
\[
\Gamma \coloneqq A^{G\setminus TE} \times \prod_{t \in T}H_t .
\]
\par 
We can check that $\tau(A^G) = \tau(\Gamma)$ (cf. the proof of \cite[Proposition~6.6]{cscp-alg-goe}). Therefore, we find that 
\begin{align*}
\mdim_\FF(\tau(A^G)) &= \mdim_\FF(\tau(\Gamma)) \\
&\leq \mdim_\FF(\Gamma) && \text{(by \cite[Proposition~5.1]{cscp-alg-goe})} \\ 
&< \dim(X) && \text{(by Lemma~\ref{l:inv-mdim-non-max})} ,
\end{align*}
which contradicts the surjectivity of $\tau$. 
Observe  that the hypothesis of Lemma~\ref{l:inv-mdim-non-max} is satisfied 
since we have $\dim H_t < \dim A^E$ for all $t \in T$.  
\par 
\underline{\textbf{Case 2}}: $\dim H = \dim A^E$. 
Then we distinguish two subcases according to whether $\dim \tau((A^E)_e)= \dim A^E$ as follows: 
\par 

\underline{\textbf{Case 2a}}: $\dim \tau((A^E)_e) <  \dim A^E$. Then Lemma~\ref{l:moore-lamma-ca-algr-full-1} tells us that there exists a proper closed subset $Z \subset A^E$ such that $\dim Z < \dim A^E$ and that $\tau((A^E)_e)=\tau(Z_e)$. We are thus in the situation of {Case 1} and obtain a contradiction. 
\par 
\underline{\textbf{Case 2b}}:  $\dim \tau((A^E)_e) =  \dim A^E$. Hence, we deduce that 
$$ 
\dim \tau((A^E)_e) = \dim \tau(H_e)= \dim H=   \dim A^E.
$$
\par 
Let $V_i$, $i \in I$, be the connected components of the algebraic group $A^E$ where $I$ is a finite set.  As $H \subset A^E$ and $\dim H = \dim A^E$, we can write $H= Z \cup V$ where $V=\cup_{j \in J}V_j$ for some $J \subsetneq I$ and $Z$ is a closed subset of $A^E$ such that $\dim Z < \dim A^E$. We find that 
$$
\tau(H_e) = \tau(Z_e) \cup \tau(V_e) 
$$
\par 
Remark that $\tau(H_e)= \tau((A^E)_e)$ is an algebraic group, all of its connected components are therefore irreducible and have the same dimension. But since  $\dim \tau(Z_e) \leq \dim Z < \dim A^E= \dim \tau(H_e)$, we deduce immediately that $\tau(H_e) = \tau(V_e)$. 
\par 
Let us consider the induced cellular automaton $\tau_0 \colon X_0^G \to X_0^G$ where the alphabet $X_0$ is the set of connected components of $X$. 
Let $\varepsilon \in X_0$ denote the connected component of $X$ containing $e$. 
We claim that $\tau_0$ is not pre-injective. Indeed,  since $J \subsetneq I$ and
$$
\tau((A^E)_e)= \tau(H_e) = \tau(V_e) = \tau ((\cup_{j \in J}V_j)_e),
$$ 
we find that $\tau_0((X_0^E)_\varepsilon)= \tau_0(Q_\varepsilon)$ where $Q \subset X_0^E$ is the set of connected components of $\cup_{j \in J}V_j$. Hence $ \vert Q\vert = \vert J \vert$. 
Since $\vert J \vert < \vert I \vert= \vert X_0^E\vert$, it follows immediately that the map $\tau_0$ is not pre-injective. 
\par
As the alphabet $X_0$ is finite and the group $G$ is amenable, we deduce from the classical Garden of Eden theorem that $\tau_0$ is not surjective. In particular, we deduce that $\tau$ is not surjective. 
Hence, we also arrive at a contradiction in this case. 
\par 
Therefore, we can conclude that $\tau$ must be $(\sbullet[.9])$-pre-injective and the proof of the theorem is complete.  
\end{proof}

\section{Reversibility in $CA_{algr}(G,X,K)$} 
\label{s:reversible}

We have seen in Theorem~\ref{t:post-surjective-implies-surjective} that post-surjectivity implies surjectivity in the classes $CA_{alg}$ and $CA_{algr}$. On the other hand, pre-injectivty is weaker than injectivity. 
As shown by Capobianco, Kari, and Taati in \cite[Theorem~1]{kari-post-surjective}, such trade-off between injectivity and surjectivity preserves bijectivity for cellular automata with finite alphabet: 

\begin{theorem}[Capobianco-Kari-Taati] 
Let $G$ be a group and let $A$ be a finite set. Then every pre-injective, post-surjective cellular automaton $\tau \colon A^G \to A^G$ is reversible. 
\end{theorem}

\par 
It turns out that the same property holds for the class $CA_{algr}$ at least in characterisitc zero. Moreover, we can show that the inverse is also an algebraic group cellular automaton. 

\begin{theorem}
\label{t:reversible}
Let $G$ be a group and let $X$ be an algebraic group over an algebraically closed field $K$ of characteristic zero. Then every post-surjective, pre-injective $\tau \in CA_{algr}(G,X,K)$ is reversible and $\tau^{-1} \in CA_{algr}(G,X,K)$. 
\end{theorem}

\begin{proof}
Suppose that $\tau \in CA_{algr}(G,X,K)$ is post-surjective and pre-injective. Let $A=X(K)$. 
Using Lemma~\ref{l:uniform-psot-surjectivity} instead of \cite[Lemma~1]{kari-post-surjective},  we have a similar result as stated in  \cite[Corollary~2]{kari-post-surjective} for the class $CA_{algr}$. 
Thus, the exact same construction given in  \cite[Theorem~1]{kari-post-surjective} shows that $\tau$ is reversible, i.e., there exists a cellular automaton $\sigma \colon A^G \to A^G$ such that $\sigma \circ \tau=\tau\circ \sigma= \Id$. In particular, $\tau$ is  bijective. 
\par 
Therefore, we can apply directly  \cite[Proposition~6.2]{phung-2020} to see that for some memory set $M\subset G$, the cellular automaton $\sigma$ admits a local defining map $A^M \to A$ which is a homomorphism of algebraic groups. It follows that $\sigma \in CA_{algr}(G,X,K)$ and the proof is complete. 
\end{proof}

\bibliographystyle{siam}

\end{document}